\documentclass{article}
\usepackage{amsmath,amsfonts,amsthm,amssymb,amscd,color,xcolor,mathrsfs, mathtools,mathrsfs,enumitem}
\usepackage{hyperref}

\colorlet{darkblue}{blue!50!black}
\colorlet{darkmagenta}{magenta!80!black}

\hypersetup{
    colorlinks,%
    citecolor=darkblue,%
    filecolor=red,%
    linkcolor=darkblue,%
    urlcolor=blue,%
    pdfnewwindow=true,%
    pdfstartview={FitH}
}

\newcommand{\p}{\partial}

\newcommand{\bS}{{\mathbb S}}
\newcommand{\R}{{\mathbb R}}
\newcommand{\Z}{{\mathbb Z}}

\newcommand{\E}{{\mathbb E}}

\newcommand{\XX}{{\cal X}}

\newcommand{\dd}{{\textup d}}

\theoremstyle{plain}

\newtheorem{theorem}{Theorem}[section]

\newtheorem{proposition}[theorem]{Proposition}
\newtheorem{corollary}[theorem]{Corollary}
\theoremstyle{definition}

\theoremstyle{remark}
\newtheorem{remark}[theorem]{Remark}

\numberwithin{equation}{section}

\begin{document}
\title{Exponential stability of the flow for a generalised Burgers equation on a circle}
\author{Ana Djurdjevac\footnote{Freie Universit\"at Berlin, Arnimallee 6, 14195 Berlin, Germany; email: \href{mailto:adjurdjevac@zedat.fu-berlin.de}{adjurdjevac@zedat.fu-berlin.de}.} \and Armen Shirikyan\footnote{Department of Mathematics, CY Cergy Paris University, CNRS UMR 8088, 2 avenue Adolphe Chauvin, 95302 Cergy--Pontoise, France \& Peoples Friendship University of Russia (RUDN University); email: \href{mailto:Armen.Shirikyan@cyu.fr}{Armen.Shirikyan@cyu.fr}.}}
\date{\today}
\maketitle

\begin{abstract}
The paper deals with the problem of stability for the flow of the 1D Burgers equation on a circle. Using some ideas from the theory of positivity preserving semigroups, we establish the strong contraction in the $L^1$ norm. As a consequence, it is proved that the equation with a bounded external force possesses a unique bounded solution on~$\R$, which is exponentially stable in~$H^1$ as $t\to+\infty$. In the case of a random  external force, we show that the difference between two trajectories goes to zero with probability~$1$. 


\end{abstract}
\tableofcontents

\section{Introduction}
Let us consider the generalised viscous Burgers equation
\begin{equation} \label{burgers}
	\p_tu-\nu\p_x^2u+\p_x f(u)=h(t,x), \quad x\in\bS. 
\end{equation}
Here $\bS=\R/\Z$ is a circle of unit length, so that all the functions are assumed to be $1$-periodic in~$x$, $\nu>0$ is a fixed parameter, $h$ is an external force, and $f\in C^2(\R)$ is a given function (called {\it flux\/}) whose second derivative satisfies the inequality\,\footnote{Inequality~\eqref{lower-bound} on the flux is not needed if the external force is a bounded function of time. In particular, Theorems~\ref{t-MT} and~\ref{t-attractor} are valid for any $C^2$-smooth flux~$f$.}
\begin{equation}\label{lower-bound}
f''(u)\ge \sigma\quad\mbox{for all $u\in\R$},
\end{equation}
where~$\sigma>0$ is a number. The function~$h$ is assumed to be locally bounded on~$\R_+\times\bS$, which ensures the well-posedness of the Cauchy problem for~\eqref{burgers}. Namely, for any initial condition~$u_0$ in the Sobolev class~$H^1(\bS)$, Eq.~\eqref{burgers} has a unique solution $u(t,x)$ in an appropriate function space such that 
\begin{equation}\label{IC}
	u(0,x)=u_0(x);
\end{equation}
see Section~\ref{s-CP} for an exact statement. To simplify the presentation, we assume in this introduction that all the functions defined on~$\bS$ have {\it zero mean value\/}. 

A simple consequence of the maximum principle is the non-expansion of the $L^1$ norm of the difference between two solutions of~\eqref{burgers}. Combining this with the regularising property of the resolving operator, it is easy to prove that the flow generated by~\eqref{burgers}, \eqref{IC} is $H^1$-stable in the following sense: if a sequence of initial conditions $(u_0^n)$ converges to~$u_0$ in~$H^1$, then the corresponding solutions~$(u^n)$ and~$u$ are such that 
$$
\sup_{t\ge0}\|u^n(t)-u(t)\|_{H^1}\to0\quad\mbox{as $n\to\infty$}. 
$$
The goal of this paper is to prove that the flow is in fact exponentially asymptotically stable as $t\to\infty$. We shall show, in particular, that the two properties below are true (see Section~\ref{s-applications} for more details):
\begin{description}
	\item [Global attractor:] \sl If $h$ is a bounded function of $t\in\R$ with range in~$H^1$, then there is a unique $H^1$-bounded solution of~\eqref{burgers} defined throughout the real line, and any other solution converges to it in~$H^1$ exponentially fast as $t\to+\infty$. 
	\item [Stochastic equation:] If $h$ is a pathwise continuous $H^2$-valued stochastic process (not necessarily bounded in time, but satisfying a mild growth condition),  then the $L^1$ norm of the difference between two solutions of Eq.~\eqref{burgers} converges to zero with probability~$1$ as $t\to+\infty$. 
\end{description}
The proof of both results is based on a squeezing property of the flow established in Section~\ref{s-MR}. The latter uses an idea coming from the theory of positive operators and applied in~\cite{shirikyan-jep2017} in the case of the Dirichlet  boundary condition. 

Let us mention that various results on the stability of the Burgers flow were obtained earlier in both deterministic and stochastic settings, for the  periodic and Dirichlet boundary conditions, as well as for the whole space. One of the first mathematically rigorous results was obtained by Sinai~\cite{sinai-1991}. He studied the cases when~$h$ is a time periodic deterministic function or spatially smooth white noise and proved the almost sure convergence of trajectories with random initial data to a limiting periodic solution or a measure, both of them independent from the initial condition. Kifer~\cite{kifer-1997} obtained a similar result for the Burgers equation in~$\R^d$ under the hypothesis that both the solution and external force are potential fields. Hill and~S\"uli~\cite{HS-1995} studied viscous scalar conservation laws in a bounded domain of~$\R^d$ with a time-independent external force and proved the existence, uniqueness, and exponential stability of a stationary solution. Jauslin--Kreiss--Moser~\cite{JKM-1999} studied the Burgers equation on~$\bS$ with a time-periodic external force and proved the existence and asymptotic stability of a time-periodic solution. Boritchev~\cite{boritchev-2013} investigated generalised Burgers equation on~$\bS$ and assuming that~$h$ is a spatially smooth white noise proved that there is a unique stationary measure and that any other solution converges to it with a rate independent of the viscosity~$\nu>0$. Chung and Kwon~\cite{CK-2016} considered the Burgers equation on~$\R$ with a quadratic flux and a non-negative time independent external force~$h$ and proved the existence and uniqueness of a non-negative stationary solution and convergence to it with an algebraic rate. Kalita and Zgliczy\'nski~\cite{KZ-2020} studied the same equation with periodic and Dirichlet boundary conditions and established the convergence to a limiting trajectory. They also proved the exponential rate of convergence in the case of the Dirichlet boundary condition. Bakhtin--Li~\cite{BL-2019} and Dunlap--Graham--Ryzhik~\cite{DGR-2021} constructed space-time stationary solutions for the stochastic Burgers equation on~$\R$ and proved their stability as $t\to+\infty$. Finally, the recent article~\cite{DR-2022} proves exponential synchronisation of a scalar viscous conservation law in a bounded domain with the Dirichlet boundary condition, in both deterministic and stochastic settings. 

The present article establishes the exponential rate of convergence in the periodic setting and simplifies many of the proofs obtained earlier. Let us also note that the results of this paper can be extended to viscous scalar conservation laws in higher dimension. However, since the corresponding proofs do not involve any new ideas and can be carried out using the methods of the articles\,\footnote{The preprint~\cite{DS-2022} was withdrawn from publication, since the main result obtained there is essentially covered by the papers~\cite{HS-1995,DR-2022}.}~\cite{DS-2022,DR-2022}, we do not discuss them in this paper.

\smallskip
The paper is organised as follows. Section~\ref{s-preliminaries} gathers various results on the Cauchy problem for~\eqref{burgers} and on linear parabolic equations. In Section~\ref{s-MR}, we formulate and prove the main theorem of this paper on the strong $L^1$-contraction of the flow of the Burgers equation. Finally, in Section~\ref{s-applications}, we present exact formulation of the two properties mentioned above and give their proofs. 

\subsubsection*{Acknowledgments}
The research of the first author has been partially supported by Deutsche Forschungsgemeinschaft (DFG) through grant CRC 1114 {\it Scaling Cascades in Complex Systems\/}, Project Number 235221301, Project C10 -- Numerical analysis for nonlinear SPDE models of particle systems. The research of the second author has been supported by the \textit{CY Initiative\/} through the grant {\it Investissements d'Avenir\/} ANR-16-IDEX-0008 and by the Ministry of Science and Higher Education of the Russian Federation (megagrant agreement No.~075-15-2022-1115).

\subsubsection*{Notation and conventions}
We write $\bS=\R/\Z$ and identify any function on~$\bS$ with a $1$-periodic function on~$\R$. The usual Lebesgue and Sobolev spaces on a domain~$D$ are denoted by~$L^p(D)$ and~$H^s(D)$, and we write $|\cdot|_p$ and~$\|\cdot\|_s$ for the corresponding norms and~$(\cdot,\cdot)$ for the $L^2$ inner product. In the case when $D=\bS$, we often drop~$\bS$ and write simply~$L^p$ and~$H^s$. Given a function $g\in L^1(\bS)$, we set 
$$
\langle g\rangle=\int_\bS g(x)\dd x. 
$$
For a closed (not necessarily bounded) interval $J\subset\R$ and a separable Banach space~$X$, we denote by~$C_b(J,X)$ the space of bounded continuous functions $f:J\to X$ endowed with the natural norm and by $L_{\rm{loc}}^2(J,X)$ the space of Borel measurable functions $f:J\to X$ such that, for any bounded interval $I\subset J$,
$$
\|f\|_{L^2(I,X)}^2:=\int_I\|f(t)\|_X^2\dd t<\infty. 
$$
Let $\XX(J)$ be the space of functions $f\in L_{\rm{loc}}^2(J,H^2)$ such that $\p_tf\in L_{\rm{loc}}^2(J,L^2)$. In the case when~$J$ is bounded, $\XX(J)$ is a Hilbert space with a norm $\|\cdot\|_\XX$ defined by 
$$
\|f\|_\XX^2=\int_J\bigl(\|f(t)\|_2^2+|\p_tf(t)|_2^2\bigr)\,\dd t.
$$
If $J=J_T:=[0,T]$, then we write $\XX_T$ instead of~$\XX(J_T)$, and if $J=\R_+$, then we write~$\XX$. Finally, for a Banach space~$X$, we denote $B_X(R)$ the closed ball in~$X$ of radius~$R$ centred at zero. 

\section{Preliminaries}
\label{s-preliminaries}

\subsection{Cauchy problem} 
\label{s-CP}
Let us fix $\nu>0$ and~$T>0$, and consider Eq.~\eqref{burgers} in the domain $Q_T:=J_T\times\bS$. The following result is well known and can be established by standard methods based on a priori estimates and a fixed point argument; see~\cite{lions1969,evans2010}. 

\begin{theorem}\label{t-WP}
	Let $f\in C^2(\R)$ be an arbitrary function, let $u_0\in H^1(\bS)$, and let $h\in L^\infty(Q_T)$. Then problem~\eqref{burgers}, \eqref{IC} has a unique solution $u\in\XX_T$, and there is a number $K_T>0$ depending on~$\nu$, $f$, $|u_0|_{L^\infty}$, and $|h|_{L^\infty(Q_T)}$ such that\,\footnote{Note that we did not impose any condition on the mean value of~$h$ with respect to~$x$, so that the norm of solutions may indeed grow with~$T$.}
	\begin{align}
	|u|_{L^\infty(Q_T)}&\le |u_0|_{L^\infty(\bS)}+T\,|h|_{L^\infty(Q_T)}	,
		\label{apriori-Linfty}\\
	\|u\|_{\XX_T}&\le K_T\bigl(\|u_0\|_1+|h|_{L^2(Q_T)}\bigr).\label{apriori-sobolev}
	\end{align}
\end{theorem}

\begin{proof}
	For the reader's convenience, we sketch the derivation of the a priori bounds. To prove~\eqref{apriori-Linfty}, we note that the function~$u(t,x)$ can be regarded as a solution of the linear equation
$$
\p_tu-\nu\p_x^2u+b(t,x)\p_xu=h(t,x),
$$
where $b(t,x)=f'(u(t,x))$. Applying the maximum principle to this equation (see Section~3.1 in~\cite{landis1998}), we obtain~\eqref{apriori-Linfty}. 
	
	To prove~\eqref{apriori-sobolev}, we first take the $L^2$ inner product of~\eqref{burgers} with~$2u$. Using the periodicity condition and the Cauchy--Schwarz inequality, we derive 
\begin{equation}\label{L2-estimate}
	\p_t|u|_2^2+2\nu\,|\p_xu|_2^2=2(h,u)\le |h|_2^2+|u|_2^2.
\end{equation}
By the Gronwall inequality, it follows that 
	$$
	\sup_{0\le t\le T}\biggl(|u(t)|_2^2+2\nu\int_0^t|\p_xu(s)|_2^2\dd s\biggr)
	\le C(T)\bigl(|u_0|_2^2+|h|_{L^2(Q_T)}^2\bigr),
	$$
	where $C(T)>0$ depends only on~$T$. To obtain bounds on higher Sobolev norms, we take the~$L^2$ inner product of~\eqref{burgers} with~$-2\p_x^2u$. After some simple transformations, we obtain
$$
\p_t|\p_xu|_2^2+2\nu\,|\p_x^2u|_2^2\le 2|h|_2|\p_xu|_2+K_1(T)|\p_xu|_2|\p_x^2u|_2\,;
$$
here and henceforth we denote by $K_i(T)$ some numbers that can be expressed in terms of the quantity $|f'(u)|_{L^\infty(Q_T)}$ (which is finite in view of~\eqref{apriori-Linfty}). Application of the Cauchy inequality results in 
$$
\p_t|\p_xu|_2^2+\nu\,|\p_x^2u|_2^2\le K_2(T)|\p_xu|_2^2+\nu^{-1}|h|_2^2. 
$$
Using again the Gronwall inequality, 	we obtain
	\begin{equation}\label{estimate-ptu}
	\sup_{0\le t\le T}\biggl(|\p_xu(t)|_2^2+\nu\int_0^t|\p_x^2u(s)|_2^2\dd s\biggr)
	\le K_3(T)\bigl(\|u_0\|_1^2+|h|_{L^2(Q_T)}^2\bigr). 	
	\end{equation}
Finally, resolving~\eqref{burgers} with respect to~$\p_tu$ and taking the $L^2$ norm over~$Q_T$, we derive
$$
|\p_tu|_2\le |h|_2+K_4(T)\,|\p_xu|_2+\nu\,|\p_x^2u|_2. 
$$
Using~\eqref{estimate-ptu}, we conclude that~$|\p_tu|_2$ can also be estimated by the right-hand side of~\eqref{apriori-sobolev}. This completes the proof of the theorem.
\end{proof}

\subsection{Dissipativity estimates}
In this section, we assume that $h:\R_+\times\bS\to\R$ is a locally integrable function such that 
\begin{equation}\label{mean-value}
\langle h(t,\cdot)\rangle=0 
\end{equation}
for almost every $t\in\R$. Let us set $Q=\R_+\times\bS$. Our first result concerns the dissipativity of the dynamics for~\eqref{burgers}. 

\begin{theorem}\label{t-dissipation}
In addition to the above hypothesis, assume that $h\in L^\infty(Q)$. Then, for any $c\in\R$, there exists a number~$C>0$ depending also on~$\nu$, $f$, and~$|h|_{L^\infty(Q)}$ such that, for any initial condition $u_0\in H^1(\bS)$ with $\langle u_0\rangle=c$, the solution $u\in\XX$ of problem~\eqref{burgers}, \eqref{IC} satisfies the following inequality for $t\ge C\ln(|u_0|_2+2)${\rm:} 
\begin{equation}\label{uniform-bound}
	\|u(t)\|_1^2+\int_t^{t+1}\bigl(\|u(s)\|_2^2+|\p_su(s)|_2^2\bigr)\dd s\le C. 
\end{equation}
\end{theorem}

\begin{proof}
Let us note that integrating~\eqref{burgers} in $t\in J_T$ and $x\in\bS$, we obtain $\langle u(T)\rangle=\langle u_0\rangle=c$ for any $T\ge0$. Substituting $u=c+v$ in~\eqref{burgers}, we obtain for~$v$ an equation of exactly the same form as for~$u$, with the function~$f$ replaced by its translation. Therefore we can assume from the very beginning that $\langle u_0\rangle =0$ and consider solutions with zero mean value. Moreover, since~$\XX_1$ is continuously embedded into~$C(J_1,H^1)$ (see Theorem~3.1 in \cite[Chapter I]{LM1972}), it suffices to estimate the integral in the left-hand side of~\eqref{uniform-bound}. 

To this end, we note that the inner product in relation~\eqref{L2-estimate} does not exceed $\nu|\p_xu|_2^2+C_1\nu^{-1}|h|_2^2$, where~$C_i$ stand for some unessential numbers not depending on the solution.  Combining this with~\eqref{L2-estimate}, we derive
\begin{equation} \label{pre-gronwall}
\p_t|u|_2^2+\nu\,|\p_xu|_2^2\le C_1\nu^{-1}|h|_2^2.
\end{equation}
Recalling that~$u$ has zero mean value in~$x$, applying the Poincar\'e inequality to~$|\p_xu|_2^2$, and using the Gronwall inequality, we derive 
$$
|u(t)|_2^2\le e^{-\gamma t}|u_0|_2^2+C_2\sup_{s\ge 0}|h(s)|_2^2. 
$$
Combining this with~\eqref{pre-gronwall}, we see that 
$$
	\int_t^{t+1}|\p_x u(s)|_2^2\,\dd s\le C_3\Bigl(e^{-\gamma t}|u_0|_2^2+\sup_{s\ge 0}|h(s)|_2^2\Bigr). 
$$
In particular, we can find a number $C_4>0$ such that 
$$
\int_t^{t+1}|\p_x u(s)|_2^2\,\dd s\le C_4\quad 
\mbox{for $t\ge T:=C_4\ln(|u_0|_2+2)$}. 
$$
Since $u$ is a continuous function of time with range in~$H^1$, for any $t\ge T+1$ we can find $t_0\in[t-1,t]$ such that $\|u(t_0)\|_1\le C_5$. Applying~\eqref{apriori-sobolev} to the Cauchy problem for~\eqref{burgers} studied on the interval $[t_0,t_0+2]$, we see that the integral in~\eqref{uniform-bound} is bounded by an absolute constant~$C$. This completes the proof of the theorem. 
\end{proof}

The next result is useful when the right-hand side of~\eqref{burgers} is unbounded. It is related to the strong nonlinear dissipation of the Burgers equation with a strictly convex flux. 

\begin{theorem}\label{t-damping}
	In addition to the hypotheses of Theorem~\ref{t-dissipation}, let us assume that $h\in C_b(\R_+,H^2)$ and~$f$ satisfies~\eqref{lower-bound}. Then, for any $c\in\R$, there is a number~$C>0$ depending also on~$\nu$, $f$, and~$|h|_{L^\infty(\R_+,H^2)}$ such that, for any initial condition $u_0\in H^1$ with $\langle u_0\rangle=c$, the solution $u\in\XX$ of problem~\eqref{burgers}, \eqref{IC} satisfies inequality~\eqref{uniform-bound} for $t\ge1$.  
\end{theorem}

\begin{proof}
	By the Kruzhkov maximum principle (see~\cite{kruzhkov-1969} or~\cite[Section~4.1]{boritchev-2013}), there is a number $C_1>0$ depending only on~$\nu$, $f$, and $h$ such that, for any $t_0\ge1/2$, we have 
$$
|u(t_0)|_{\infty}\le C_1. 
$$
In view of inequality~\eqref{apriori-Linfty} (which remains valid for $L^\infty$ initial conditions), the norm of the solution in the space $L^\infty((t_0,t_0+1)\times\bS)$ is bounded by a constant depending only on~$h$. Taking the $L^2$ inner product of~\eqref{burgers} with the function $-(t-t_0)\p_x^2u$ and carrying out some simple transformations, we easily obtain a universal bound on the $H^1$ norm of the solution at $t=t_0+1/2$, so that 
$$
\sup_{t\ge1}\|u(t)\|_1\le C_2,
$$
where $C_2>0$ depends on~$\nu$, $f$, and~$h$. The proof can now be completed by exactly the same argument as for Theorem~\ref{t-dissipation}. 
\end{proof}

Combining Theorems~\ref{t-dissipation} and~\ref{t-damping} with the regularising property of~\eqref{burgers}, we arrive at the following bound for higher derivatives of solutions. Its proof can be obtained by taking the~$L^2$ inner product of~\eqref{burgers} with the functions $(t-t_0)\p_x^4u$ or~$\p_x^4u$ and carrying out some standard transformations. 

\begin{corollary}\label{c-H2bound}
	Under the hypotheses of Theorem~\ref{t-dissipation} (or those of Theorem~\ref{t-damping}), there is a number $C>0$ not depending on the initial condition such that the solution $u\in\XX$ of~\eqref{burgers}, \eqref{IC} satisfies the inequality
	\begin{equation}\label{H2-bound}
		\|u(t)\|_2\le C,
	\end{equation}
	where $t\ge 0$ varies in the same half-lines as in the above theorems. If, in addition, the initial condition~$u_0$ belongs to~$H^2$, then~\eqref{H2-bound} holds for all $t\ge0$ with some number~$C>0$ depending also on~$\|u_0\|_2$. 
\end{corollary}

\subsection{Some properties of linear parabolic equations}
\label{s-harnck}
A key point of our proof of the stability of the Burgers flow is a lower bound for positive solutions of parabolic equations. The corresponding result is expressed in terms of the Harnack inequality for the linear equation
\begin{equation}\label{parabolic-pde}
	\p_tw-\nu\p_x^2w+\p_x\bigl(a(t,x)w\bigr)=0,\quad (t,x)\in Q_T,
\end{equation}
where $a:Q_T\to\R$ is a given function. A proof of the following proposition can be found in~\cite{KS-1980} (see also Section~IV.2 in~\cite{krylov1987}).  

\begin{proposition}\label{p-harnack}
	Let $\nu$, $T$, and~$\rho$ be positive numbers and let $a\in L^\infty(Q_T)$ be a function such that $\p_xa\in L^\infty(Q_T)$ and 
	\begin{equation}\label{a-bound}
		|a|_{L^\infty(Q_T)}+|\p_xa|_{L^\infty(Q_T)}\le\rho.
	\end{equation}
	Then, for any $T'\in(0,T)$, there is $\theta=\theta(T',T,\nu,\rho)>0$ such that, for any non-negative solution $w\in\XX_T$ of~\eqref{parabolic-pde}, we have  
	\begin{equation}\label{harnack-inequality}
\theta\max_{x\in\bS}w(T',x)\le \min_{x\in\bS}w(T,x). 
	\end{equation}
\end{proposition}

We shall also need the property of non-expansion of the $L^1$ norm of solutions for~\eqref{parabolic-pde}. The latter  follows immediately from the maximum principle by a duality argument; see Lemma~3.2.2 in~\cite{hormander1997}.

\begin{proposition}\label{p-L1nonexp}
	Under the hypotheses of Proposition~\ref{p-harnack}, for any solution $w\in \XX_T$ of~\eqref{parabolic-pde}, we have 
	\begin{equation}\label{L1-nonexp}
		|w(t)|_1\le |w(s)|_1\quad\mbox{for $0\le s\le t\le T$}. 
	\end{equation}
\end{proposition}

\section{Main result}
\label{s-MR}
The main tool for proving the stability of the Burgers flow is the contraction of the~$L^1$ norm of the difference between two solutions. The rate of contraction depends on the size of solutions, which in turn is determined by the initial conditions and the right-hand side. More precisely, we have the following result. 

\begin{theorem}\label{t-MT}
	Let $f\in C^2(\R)$ be an arbitrary function. Then, for any $c\in\R$ and any  positive numbers $\nu$, $R$, and~$T$, there is $q\in(0,1)$ such that the following property holds: for any external force $h\in L^\infty(Q_T)$, with an $L^\infty$ norm not exceeding~$R$ and mean value satisfying~\eqref{mean-value} for a.\,e.~$t\in J_T$, and any initial conditions $u_0,v_0\in B_{H^1}(R)$ with $\langle u_0\rangle=\langle v_0\rangle=c$, the corresponding solutions satisfy the inequality 
	\begin{equation}\label{q-contraction}
	|u(T)-v(T)|_1\le q\,|u_0-v_0|_1.
	\end{equation}
\end{theorem} 

\begin{proof}
Let us write~$w$ for the difference $u-v$ and note that it satisfies Eq.~\eqref{parabolic-pde}, in which 
$$
a(t,x)=\int_0^1 f'(v(t,x)+\tau w(t,x))\,\dd \tau. 
$$
Moreover, we have $\langle w(t)\rangle=0$ for any $t\in J_T$, and in view of inequality~\eqref{H2-bound} applied to~$u$ and~$v$, there is $\rho>0$ depending only on~$\nu$, $f$, $R$, and~$T$ such that~\eqref{a-bound} holds. In particular, the Harnack inequality~\eqref{harnack-inequality} is valid for any non-negative solution of~\eqref{parabolic-pde}. 

Let us set $w_0=u_0-v_0$ and write~$w_0^+$  and~$w_0^-$ for the positive and negative parts of~$w_0$; that is, $w_0^\pm=\max\{\pm w_0, 0\}$. We denote by~$w^\pm$ the solution of~\eqref{parabolic-pde} that is equal to~$w_0^\pm$ at time $t=0$, so that $w(t)=w^+(t)-w^-(t)$ for any $t\in J_T$ and~$w^{\pm}$ are non-negative functions. If $w^+(T/2,x)\le \frac14|w_0|_1$ for any $x\in\bS$, then $|w^+(T/2)|_1\le \frac14|w_0|_1$. Since the mean value of $w(t)$ is zero, the $L^1$ norms of $w^+(t)$ and $w^-(t)$ are the same, so that $|w^-(T/2)|_1\le \frac14|w_0|_1$. Combining this with~\eqref{L1-nonexp}, we derive
$$
|w(T)|_1\le |w(T/2)|_1
\le |w^+(T/2)|_1+|w^-(T/2)|_1\le \frac12\,|w_0|_1. 
$$
Exactly the same argument applies when $w^-(T/2,x)\le \frac14|w_0|_1$ for any $x\in\bS$. 

We now assume that 
$$
\max_{x\in\bS}w^\pm(T/2,x)\ge\frac14\,|w_0|_1. 
$$
In view of the Harnack inequality~\eqref{harnack-inequality}, there is a number $\theta>0$ such that
$$
\min_{x\in\bS}w^\pm(T,x)\ge \theta\max_{x\in\bS}w^\pm(T/2,x)\ge\frac{\theta}{4}\,|w_0|_1. 
$$
Combining this with~\eqref{L1-nonexp}, we derive
\begin{align*}
	|w(T,x)|_1&=\int_\bS \bigl|w^+(T,x)-w^-(T,x)\bigr|\,\dd x\\
	&=\int_\bS \bigl|(w^+(T,x)-\tfrac{\theta}{4}|w_0|_1)
	-(w^-(T,x)-\tfrac{\theta}{4}|w_0|_1)\bigr|\,\dd x\\
	&\le\int_\bS \bigl(w^+(T,x)-\tfrac{\theta}{4}|w_0|_1\bigr)\,\dd x
	+\int_\bS \bigl(w^-(T,x)-\tfrac{\theta}{4}|w_0|_1\bigr)\,\dd x\\
	&=|w^+(T)|_1+|w^-(T)|_1-\tfrac{\theta}{2}\,|w_0|_1\\
	&\le |w_0^+|_1+|w_0^-|_1-\tfrac{\theta}{2}\,|w_0|_1
	=\bigl(1-\tfrac{\theta}{2}\bigr)|w_0|_1.
\end{align*}
We thus obtain the required inequality~\eqref{q-contraction} with $q=\max\{\frac12,1-\frac\theta2\}$. 
\end{proof}

\section{Applications}
\label{s-applications}

\subsection{Uniqueness and stability of a bounded trajectory}
Let us consider Eq.~\eqref{burgers}, in which $f\in C^2(\R)$ is an arbitrary  function, and~$h$ belongs to the space $C_b(\R,H^1)$  and satisfies~\eqref{mean-value} for all $t\in\R$. The following result describes the large-time behaviour of solutions for~\eqref{burgers}. 

\begin{theorem}\label{t-attractor}
	Under the above hypotheses, for any $c\in\R$, there is a unique solution~$v(t,x)$ of~\eqref{burgers} in the space $\XX(\R)\cap C_b(\R,H^1)$ such that 
\begin{equation}\label{mean-value-solution}
\langle v(t,\cdot)\rangle=c\quad\mbox{for all $t\in\R$}. 
\end{equation}
	Moreover, there exists $\gamma>0$ such that, for any $R>0$, a sufficiently large $C_R>0$, and any initial condition $u_0\in B_{H^1}(R)$ with $\langle u_0\rangle=c$, the corresponding solution~$u(t,x)$ satisfies the inequality 
\begin{equation}\label{expo-conv}
	\|u(t)-v(t)\|_1\le C_R\,e^{-\gamma t}|u_0-v(0)|_1^{2/5}, \quad t\ge1.
\end{equation}
\end{theorem}

\begin{proof}
{\it Step~1: Exponential stability\/}. We first prove that, for any $R>0$ and any initial conditions $u_{01},u_{02}\in B_{H^1}(R)$ with $\langle u_{01}\rangle=\langle u_{02}\rangle$, the corresponding solutions satisfy the inequality
\begin{equation}\label{expodecay-difference}
	\|u_1(t)-u_2(t)\|_1\le C_R\,e^{-\gamma t}|u_{01}-u_{02}|_1^{2/5}, \quad t\ge1,
\end{equation}
where $\gamma>0$ does not depend on the initial conditions. To this end, we first consider the case when the solutions satisfy inequality~\eqref{H2-bound} for all $t\ge0$. By a well-known interpolation inequality (see Section~15.1 in~\cite{BIN1979}), for any $u\in H^2$ we have $\|u\|_1\le C_1\|u\|_2^{3/5}|u|_1^{2/5}$. Since the $H^2$ norm of the difference $u=u_1-u_2$ is bounded, inequality~\eqref{expodecay-difference} will be established for $t\ge0$ once we prove the exponential decay of the $L^1$ norm $|u(t)|_1$. 

To see this, let us note that the difference $u(t,x)$ satisfies the linear equation~\eqref{parabolic-pde}. In view of the non-expansion of the $L^1$ norm (see Proposition~\ref{p-L1nonexp}), it suffices to prove that $|u(k)|_1\le q^k|u(0)|_1$ for any integer $k\ge0$, where $q\in(0,1)$ does not depend on~$k$. The latter estimate is an immediate consequence of Theorem~\ref{t-MT}. 

To prove~\eqref{expodecay-difference} for arbitrary initial conditions in~$B_{H^1}(R)$, note that the solutions~$u_i$ satisfy~\eqref{H2-bound} for $t\ge1$ with some constant $C=C(R)$. Using the interpolation inequality and the non-expansion of the $L^1$ norm, we see that $\|u(t)\|_1\le C_2(R)|u(0)|_1^{2/5}$ for $t\ge1$. On the other hand, the solutions $u_i\in\XX$ satisfy~\eqref{H2-bound} for $t\ge t_0=C_3\ln(R+2)$, so that 
$$
\|u(t)\|_1\le C_4e^{-\gamma(t-t_0)}|u(t_0)|_1^{2/5}
\le C_5e^{-\gamma t}|u(t_0)|_1^{2/5}
$$ 
for $t\ge t_0$. Combining this with the above estimate for $\|u(t)\|_1$ valid for $t\ge1$, we arrive at~\eqref{expodecay-difference}. 

\smallskip
{\it Step~2: Bounded solution\/}. We now fix $c\in\R$ and construct an $H^1$-bounded solution $v\in\XX(\R)$ of~\eqref{burgers} defined throughout the real line such that~\eqref{mean-value-solution} holds. Once it is done, the stability inequality~\eqref{expo-conv} will follow from~\eqref{expodecay-difference}. 

The construction of~$v$ is based on a standard argument using solutions issued from  initial times going to~$-\infty$. Namely, let us denote by $u^n$ the solution of~\eqref{burgers} satisfying the initial condition $u^n(-n)=c$. Then, for any $m<n$ and $T\ge 1-m$, inequality~\eqref{expodecay-difference} applied on the half-line $[-m,+\infty)$ implies that 
$$
\sup_{|t|\le T}\|u^m(t)-u^n(t)\|_1\le C_6e^{-\gamma (T+m)}. 
$$
Letting $m\to+\infty$, we see that $(u^n)$ converges in~$H^1$ to a limit $v\in C_b(\R,H^1)$, uniformly on any bounded interval. Furthermore, the sequence~$(u^n)$ is bounded in~$\XX([-T,T])$ for any $T>0$, and the weak compactness of a unit ball in a Hilbert space implies that $v\in \XX(\R)$. It is straightforward to see that $v$ is a solution of~\eqref{burgers} and that~\eqref{mean-value-solution} holds. This completes the proof of the theorem. 
\end{proof}

\subsection{Equation with random external force}
We now consider the case when the external force~$h$ is a stochastic process. Namely, we assume that almost every trajectory of~$h$ is a continuous function of time with range in~$H^2$ such that relation~\eqref{mean-value} holds for $t\ge0$, and the random variable 
\begin{equation}\label{h-bound}
K:=\limsup_{T\to\infty}\frac1T\int_0^T\max_{t\le s\le t+1}\|h(s)\|_2\,\dd t
\end{equation}
is almost surely finite. For instance, in view of the Birkhoff theorem, condition~\eqref{h-bound} is certainly satisfied if~$h$ is a stationary $H^2$-valued stochastic process with almost surely continuous trajectories such that 
\begin{equation}\label{M-finite}
	M:=\E\max_{0\le t\le 1}\|h(t)\|_2<\infty.
\end{equation}

\begin{theorem}\label{t-stochastic}
	Let $f\in C^2(\R)$ be an arbitrary function satisfying~\eqref{lower-bound}, let the above hypotheses be fulfilled for~$h$, and let $u_0,v_0\in H^1$ be arbitrary initial conditions such that $\langle u_0\rangle=\langle v_0\rangle$. Then, with probability~$1$, the corresponding solutions belong to~$\XX$ and are such that 
	\begin{equation}\label{random-convergence}
	|u(t)-v(t)|_1\to0\quad\mbox{as $t\to\infty$}. 
	\end{equation}
\end{theorem}

\begin{remark}
Theorem~\ref{t-stochastic} does not  apply to the case when the external force in~\eqref{burgers} has the form $h+\eta$, where $h=h(x)$ is a deterministic function, and $\eta=\eta(t,x)$ is a spatially regular white noise. However, combining dissipativity of the dynamics with the Markov property, it is not difficult to construct an increasing sequence of stopping times $(t_k)$ such that the difference $t_k-t_{k-1}$ has a finite exponential moment, and the solutions are bounded in~$H^2$ on $[t_k,t_k+1]$. Repeating then the argument used in the proof of Theorem~\ref{t-stochastic}, one can establish the exponential convergence in~\eqref{random-convergence} and the uniqueness of a stationary measure. Since the approach described above is well known and is carried out in detail in the paper~\cite[Section~3]{DR-2022}, we skip the corresponding argument. 
\end{remark}

\begin{proof}[Proof of Theorem~\ref{t-stochastic}]
	Let us fix two initial conditions $u_0,v_0\in B_{H^1}(R)$ with the same mean value and denote by~$u$ and~$v$ the corresponding solutions. In view of Theorem~\ref{t-WP}, they are well defined and belong to~$\XX$ with probability~$1$. It follows from~\eqref{h-bound} that there is an increasing random sequence $(t_k)$ going to~$+\infty$ with probability~$1$ such that 
$$
\sup_{t_k\le t\le t_k+2}\|h(t)\|_2\le 2K+1. 
$$
	Applying Corollary~\ref{c-H2bound} to the solutions~$u$ and~$v$ on the intervals $[t_k,t_k+2]$, we find an almost surely finite random constant $C>0$ such that 
	$$
	\sup_{t_k+1\le t\le t_k+2}(\|u(t)\|_2+\|v(t)\|_2)\le C. 
	$$
	Using now Theorem~\ref{t-MT}, we can construct a random variable $q\in(0,1)$ such that $|u(t_k+2)-v(t_k+2)|_1\le q\,|u(t_k+1)-v(t_k+1)|_1$ for any $k\ge1$. Combining this with~\eqref{L1-nonexp}, we conclude that, with probability~$1$, 
	$$
	|u(t_k+2)-v(t_k+2)|_1\le q^k|u_0-v_0|_1 \quad \mbox{for all $k\ge1$}. 
	$$
	Using again~\eqref{L1-nonexp}, we see  that~\eqref{random-convergence} holds almost surely. 
\end{proof}

\begin{remark}
Theorem~\ref{t-stochastic} does not specify the rate of convergence in~\eqref{random-convergence}, and it is unlikely that one can say something more about~\eqref{random-convergence} without any additional information. On the other hand, if~$h$ is a mixing stationary process such that~\eqref{M-finite} holds, then the random variable~$K$ in~\eqref{h-bound} is almost surely equal to~$M$ (by the Birkhoff theorem). In this case, any information about the rate of mixing will provide some quantitative estimates for the random times $(t_k)$ used in the above proof, and this allows to describe the rate of convergence in~\eqref{random-convergence}. 
\end{remark}

\addcontentsline{toc}{section}{Bibliography}

\def\cprime{$'$} \def\cprime{$'$}
  \def\polhk#1{\setbox0=\hbox{#1}{\ooalign{\hidewidth
  \lower1.5ex\hbox{`}\hidewidth\crcr\unhbox0}}}
  \def\polhk#1{\setbox0=\hbox{#1}{\ooalign{\hidewidth
  \lower1.5ex\hbox{`}\hidewidth\crcr\unhbox0}}}
  \def\polhk#1{\setbox0=\hbox{#1}{\ooalign{\hidewidth
  \lower1.5ex\hbox{`}\hidewidth\crcr\unhbox0}}} \def\cprime{$'$}
  \def\polhk#1{\setbox0=\hbox{#1}{\ooalign{\hidewidth
  \lower1.5ex\hbox{`}\hidewidth\crcr\unhbox0}}} \def\cprime{$'$}
  \def\cprime{$'$} \def\cprime{$'$} \def\cprime{$'$}
\providecommand{\bysame}{\leavevmode\hbox to3em{\hrulefill}\thinspace}
\providecommand{\MR}{\relax\ifhmode\unskip\space\fi MR }
\providecommand{\MRhref}[2]{%
  \href{http://www.ams.org/mathscinet-getitem?mr=#1}{#2}
}
\providecommand{\href}[2]{#2}


\begin{thebibliography}{DGR21}

\bibitem[BIN79]{BIN1979}
O.~V. Besov, V.~P. Il{\cprime}in, and S.~M. Nikol{\cprime}ski{\u\i},
  \emph{Integral {R}epresentations of {F}unctions and {I}mbedding {T}heorems.
  {V}ol. {I}, {II}}, V. H. Winston \& Sons, Washington, D.C., 1979.

\bibitem[BL19]{BL-2019}
Y.~Bakhtin and L.~Li, \emph{Thermodynamic limit for directed polymers and
  stationary solutions of the {B}urgers equation}, Comm. Pure Appl. Math.
  \textbf{72} (2019), no.~3, 536--619.

\bibitem[Bor13]{boritchev-2013}
A.~Boritchev, \emph{{Sharp estimates for turbulence in white-forced generalised
  Burgers equation}}, Geom. Funct. Anal. \textbf{23} (2013), no.~6, 1730--1771.

\bibitem[CK16]{CK-2016}
J.~Chung and O.~Kwon, \emph{Asymptotic behavior for the viscous {B}urgers
  equation with a stationary source}, J. Math. Phys. \textbf{57} (2016),
  no.~10, 101506, 10.

\bibitem[DGR21]{DGR-2021}
A.~Dunlap, C.~Graham, and L.~Ryzhik, \emph{Stationary solutions to the
  stochastic {B}urgers equation on the line}, Comm. Math. Phys. \textbf{382}
  (2021), no.~2, 875--949.

\bibitem[DR22]{DR-2022}
A.~Djurdjevac and T.~Rosati, \emph{Synchronisation for scalar conservation laws
  via {D}irichlet boundary}, Preprint (2022), arXiv:2211.05814.

\bibitem[DS22]{DS-2022}
A.~Djurdjevac and A.~Shirikyan, \emph{Stabilisation of a viscous conservation
  law by a one-dimensional external force}, Preprint (2022), arXiv:2204.03427.

\bibitem[Eva10]{evans2010}
L.~C. Evans, \emph{{Partial Differential Equations}}, American Mathematical
  Society, Providence, RI, 2010.

\bibitem[H{\"o}r97]{hormander1997}
L.~H{\"o}rmander, \emph{Lectures on {N}onlinear {H}yperbolic {D}ifferential
  {E}quations}, Springer-Verlag, Berlin, 1997.

\bibitem[HS95]{HS-1995}
A.~T. Hill and E.~S\"{u}li, \emph{Dynamics of a nonlinear convection-diffusion
  equation in multidimensional bounded domains}, Proc. Roy. Soc. Edinburgh
  Sect. A \textbf{125} (1995), no.~2, 439--448.

\bibitem[JKM99]{JKM-1999}
H.~R. Jauslin, H.~O. Kreiss, and J.~Moser, \emph{On the forced {B}urgers
  equation with periodic boundary conditions}, Differential equations: {L}a
  {P}ietra 1996, Amer. Math. Soc., Providence, RI, 1999, pp.~133--153.

\bibitem[Kif97]{kifer-1997}
Y.~Kifer, \emph{The {B}urgers equation with a random force and a general model
  for directed polymers in random environments}, Probab. Theory Related Fields
  \textbf{108} (1997), no.~1, 29--65.

\bibitem[Kru69]{kruzhkov-1969}
S.~N. Kru\v{z}kov, \emph{The {C}auchy problem for certain classes of
  quasilinear parabolic equations}, Mat. Zametki \textbf{6} (1969), 295--300.

\bibitem[Kry87]{krylov1987}
N.~V. Krylov, \emph{Nonlinear {E}lliptic and {P}arabolic {E}quations of the
  {S}econd {O}rder}, D. Reidel Publishing Co., Dordrecht, 1987.

\bibitem[KS80]{KS-1980}
N.~V. Krylov and M.~V. Safonov, \emph{A property of the solutions of parabolic
  equations with measurable coefficients}, Izv. Akad. Nauk SSSR Ser. Mat.
  \textbf{44} (1980), no.~1, 161--175, 239.

\bibitem[KZ20]{KZ-2020}
P.~Kalita and P.~Zgliczy\'{n}ski, \emph{On non-autonomously forced {B}urgers
  equation with periodic and {D}irichlet boundary conditions}, Proc. Roy. Soc.
  Edinburgh Sect. A \textbf{150} (2020), no.~4, 2025--2054.

\bibitem[Lan98]{landis1998}
E.~M. Landis, \emph{{Second Order Equations of Elliptic and Parabolic Type}},
  American Mathematical Society, Providence, RI, 1998.

\bibitem[Lio69]{lions1969}
J.-L. Lions, \emph{Quelques {M}\'ethodes de {R}\'esolution des {P}robl\`emes
  aux {L}imites {N}on {L}in\'eaires}, Dunod, 1969.

\bibitem[LM72]{LM1972}
J.-L. Lions and E.~Magenes, \emph{Non-{N}omogeneous {B}oundary {V}alue
  {P}roblems and {A}pplications. {V}ol. {I}}, Springer-Verlag, New York, 1972.

\bibitem[Shi17]{shirikyan-jep2017}
A.~Shirikyan, \emph{Global exponential stabilisation for the {B}urgers equation
  with localised control}, {Journal de l'\'Ecole polytechnique
  (Math\'ematiques)} \textbf{4} (2017), 613--632.

\bibitem[Sin91]{sinai-1991}
Ya.~G. Sina{\u\i}, \emph{Two results concerning asymptotic behavior of
  solutions of the {B}urgers equation with force}, J. Statist. Phys.
  \textbf{64} (1991), no.~1-2, 1--12.
\end{thebibliography}
\end{document}